\documentclass[a4paper,11pt]{article}
\usepackage{amssymb}
\usepackage{amsmath, amsthm, amscd, amsfonts, amssymb, graphicx, color}

\newtheorem{theorem}{Theorem}[section]
\newtheorem{lemma}[theorem]{Lemma}
\newtheorem{proposition}[theorem]{Proposition}
\newtheorem{corollary}[theorem]{Corollary}
\theoremstyle{definition}
\newtheorem{definition}[theorem]{Definition}
\newtheorem{example}[theorem]{Example}

\theoremstyle{remark}
\newtheorem{remark}[theorem]{Remark}
\numberwithin{equation}{section}

\title {On operators with closed range and semi-Fredholm operators over $W^{*}$-algebras}
\author{Stefan Ivkovi\'{c}}

\begin{document}
	\begin{center}
		\textbf{On operators with closed range and semi-Fredholm operators over $W^{*}$-algebras}
	\end{center}
\begin{center}
	\textbf{	Stefan Ivkovi\'{c}}
\end{center}
\begin{center}
		The Mathematical Institute of the Serbian Academy of Sciences and Arts,\\
		p.p. 367, Kneza Mihaila 36, 11000 Beograd, Serbia,\\
		E-mail: stefan.iv10@outlook.com, Tel.: +381-69-774237 
\end{center}
\begin{abstract}
In this paper we consider $\mathcal{A}$-Fredholm and semi-$\mathcal{A}$-Fredholm operators on Hilbert $C^{*}$-modules over a $W^{*}$-algebra $\mathcal{A}$ defined in \cite {I},\cite {MF}. Using the assumption that $\mathcal{A}$ is a $W^{*}$-algebra (and not an arbitrary $C^{*}$-algebra)
% we obtain several special properties such as that a product of two upper (or lower) semi-$\mathcal{A}$-Fredholm operators with closed image also has closed image, 
such as a generalization of Schechter-Lebow characterization of semi-Fredholm operators and a generalization of "punctured neighbourhood" theorem, as well as some other results that generalize their classical counterparts. We consider both adjointable and non adjointable semi-Fredholm operators over $W^{*}$-algebras. Moreover, we also work with general bounded, adjointable operators with closed ranges over $C^{*}$-algebras and prove a generalization to Hilbert $C^{*}$-modules of the result in \cite{Bld} on Hilbert spaces.
\\
\textbf{Keywords} Hilbert C*-modules, W*-algebras, semi- Fredholm operators\\
\textbf{Mathematics Subject Classification (2010)} Primary MSC 47A53; Secondary MSC 46L08\\
\end{abstract}	
\section{Introduction }
	Fredholm theory on Hilbert $\mathrm{C}^*$-modules as a generalization of Fredholm theory on Hilbert spaces was started by Mishchenko and Fomenko in \cite{MF}. They have elaborated the notion of a Fredholm operator on the standard module $H_{\mathcal{A}} $ and proved the generalization of the Atkinson theorem.In \cite{I} we went further in this direction and defined semi-Fredholm operators on Hilbert $\mathrm{C}^{*}$-modules. We proved then several properties of these generalized semi Fredholm operators on Hilbert $\mathrm{C}^{*}$-modules as an analogue or generalization of the well-known properties of classical semi-Fredholm operators on Hilbert and Banach spaces. Several special properties of $\mathcal{A}$-Fredholm operators in the case of $W^{*}$-algebra were described in \cite[Section 3.6]{MT}. The idea in this paper was to go further in this direction and establish more special properties $\mathcal{A}$-Fredholm operators defined in \cite{MF} and of semi$- \mathcal{A} $-Fredholm operator defined in  \cite {I}, in the case when $\mathcal{A}$ is a $W^{*}$-algebra, the properties that are closer related to the properties of the classical semi-Fredholm operators on Hilbert spaces than in the general case, when $\mathcal{A}$ is an arbitrary $C^{*}$-algebra. Moreover, we consider both adjointable and non-adjointable semi-Fredholm operators over $W^{*}$-algebras in this paper.\\
	Here is the list of our main results. Proposition  \ref{P03} and Lemma \ref{L06}  generalize the part of the index theorem which states that if $F,D$ are Fredholm operators on a Hilbert spaces $H$, then $\dim \ker FD \leq \dim \ker F + \dim \ker D$ and $\dim ImFD^{\bot} \leq \dim Im F^{\bot} + \dim Im D^{\bot}.$ \\
	Corollary \ref{C04} and Lemma \ref{L01} ,  and Proposition \ref{P10} is a generalization of \cite[Theorem 1.5.7]{ZZRD}, originally given in \cite{SE}.  Theorem \ref{T01}, and Corollary \ref{C03}, Lemma \ref{L05}  and Proposition \ref{P10} are analogue of Schechter's and Lebow's characterization of semi-Fredhnolm operators \cite[Theorem 1.4.4]{ZZRD}  and  \cite[Theorem 1.4.5]{ZZRD}, originally given in \cite{LS}, \cite{SC}, Theorem \ref{T02} is a generalization of the classical "punctured neighbourhood theorem"  \cite[Theorem 1.7.7]{ZZRD}, originally given in  \cite{LAY}. Compared to the classical version on Hilbert spaces, our generalization (Theorem \ref{T02}) needs additional assumption on the operator $F \in  {\mathcal{M}\Phi}(M),$ denoted by (*). It turns out that in the case of ordinary  Hilbert spaces, (*) is automatically satisfied for any Fredhnolm operator, so in the case of ordinary Hilbert spaces, Theorem 3.14 reduces to the classical "punctured neighbourhood" theorem. However, in Example \ref{E01}, we give an example of a Hilbert $C^{*}$-module over a $W^{*}$-algebra $\mathcal{A}$ which is not a Hilbert space and where the condition (*) is satisfied for all $\mathcal{A}$-Fredholm operators as long as they have closed image.\\
	In several results in this paper we consider semi-$\mathcal{A}$-Fredholm operators with closed image. Ordinary semi-Fredholm operators on Hilbert spaces have always closed image, however in our generalizations to modules we need sometimes to provide this additional assumption in order to obtain an analogue of the classical results. This led us to study in general bounded, adjointable operators over $C^{*}$-algebras with closed image and not only semi-Fredholm operators over $W^{*}$-algebra. We prove in Lemma \ref{L07} that if $F, D$ are two bounded, adjointable operators on the standard module with closed images, then $Im DF$ is closed iff the Dixmier angle between $Im F$ and $\ker D \cap (\ker D \cap Im F)^{\perp}   $ is positive, or equivalently iff the Dixmier angle between $ \ker D  $ and $Im F \cap (\ker D \cap Im F)^{\perp}$ is positive. This is an anlogue on Hilbert $C^{*}$-modules of the well known result in \cite{Bld} on Hilbert spaces. Moreover, our Lemma \ref{L06}, which generalizes the above mentioned second part of the classical index theorem, yields for arbitrary bounded, adjointable operators $F, D$ on the standard $C^{*} $ module provided that $Im F, Im D, Im DF$ are closed. Next, our Lemma  \ref{L11} gives another, simplified proof of the result in \cite{Sh}. This result follows as corollary from Lemma \ref{L11}, Corollary  \ref{C05}.  \\
	Important tools for proving most of the results in this paper are \cite[Corollary 3.6.4]{MT},\cite[Corollary 3.6.7]{MT}, \cite[Proposition 3.6.8]{MT} originally given in \cite{FM},\cite{LAN},\cite{LIN} and these results assume that $\mathcal{A}$ is a $W^{*}$-algebra. That's why we deal mainly here with Hilbert $\mathrm{C}^{*}$-modules over $W^{*}$-algebras. However, our Lemma \ref{L04},  Corollary \ref{C01}, Lemma \ref{L11}, Lemma \ref{L07}, Corollary \ref{C05}, Corollary \ref{C02} and Lemma \ref{L06} hold also in the case when $\mathcal{A} $ is an arbitrary unital $C^{*}$-algebra and not only $W^{*}$-algebra.
\section{Preliminaries }
	Throughout this paper we let $\mathcal{A}$ be a $W^{*}$-algebra, $H_{\mathcal{A}}$ or $l_{2}(\mathcal{A})$ be the standard Hilbert $C^{*}$-module over $\mathcal{A}$ and we let $B^{a}(H_{\mathcal{A}})$ denote the set of all bounded , adjointable operators on $H_{\mathcal{A}}.$ 
	Similarly, if $M$ is an arbitrary Hilbert $C^{*}$-module, we let $B^{a}(M) $ denote the set of all bounded, adjojntable  operators on $M$. We let  $B(l_{2}(\mathcal{A}))$ denote the set of all $\mathcal{A}$-linear, bounded, but not necessarily adjointable operators on $l_{2}(\mathcal{A}).$  According to \cite[ Definition 1.4.1] {MT}, we say that a Hilbert $\mathrm{C}^*$-module $M$ over $\mathcal{A}$ is finitely generated if there exists a finite set $ \lbrace x_{i} \rbrace \subseteq M $  such that $M $ equals the linear span (over $\mathrm{C}$ and $\mathcal{A} $) of this set.\\
	The notation $\tilde{ \oplus} $ denotes the direct sum of modules without orthogonality, as given in \cite{MT}.\\
	\begin{definition} \label{D01} %\textbf{D01} 
		\cite[Definition 2.1]{I}
		Let $\mathrm{F} \in B^{a}(H_{\mathcal{A}}).$ We say that $\mathrm{F} $ is an upper semi-{$\mathcal{A}$}-Fredholm operator if there exists a decomposition $$H_{\mathcal{A}} = M_{1} \tilde \oplus {N_{1}}‎‎\stackrel{\mathrm{F}}{\longrightarrow} M_{2} \tilde \oplus N_{2}= H_{\mathcal{A}} $$ with respect to which $\mathrm{F}$ has the matrix\\
	\begin{center}
		$\left\lbrack
		\begin{array}{ll}
		\mathrm{F}_{1} & 0 \\
		0 & \mathrm{F}_{4} \\
		\end{array}
		\right \rbrack,
		$
	\end{center}
	where $\mathrm{F}_{1}$ is an isomorphism $M_{1},M_{2},N_{1},N_{2}$ are closed submodules of $H_{\mathcal{A}} $ and $N_{1}$ is finitely generated. Similarly, we say that $\mathrm{F}$ is a lower semi-{$\mathcal{A}$}-Fredholm operator if all the above conditions hold except that in this case we assume that $N_{2}$ ( and not $N_{1}$ ) is finitely generated.	
	\end{definition}
	Set
	\begin{center}
		$\mathcal{M}\Phi_{+}(H_{\mathcal{A}})=\lbrace \mathrm{F} \in B^{a}(H_{\mathcal{A}}) \mid \mathrm{F} $ is upper semi-{$\mathcal{A}$}-Fredholm $\rbrace ,$	
	\end{center}
	\begin{center}
		$\mathcal{M}\Phi_{-}(H_{\mathcal{A}})=\lbrace \mathrm{F} \in B^{a}(H_{\mathcal{A}}) \mid \mathrm{F} $ is lower semi-{$\mathcal{A}$}-Fredholm $\rbrace ,$	
	\end{center}
	\begin{center}
		$\mathcal{M}\Phi(H_{\mathcal{A}})=\lbrace \mathrm{F} \in B^{a}(H_{\mathcal{A}}) \mid \mathrm{F} $ is $\mathcal{A}$-Fredholm operator on $H_{\mathcal{A}}\rbrace .$ 
	\end{center} 
	Next, we let $ K^{*}(H_{\mathcal{A}})  $ denote the set of all adjointable compact operators in the sense of \cite[Section 2.2]{MT} and we let $K(l_{2}(\mathcal{A}))  $ denote the set of all compact operators (not necessarily adjointable) in the sense of \cite{IM}. We set $\widehat{ {\mathcal{M}\Phi}}_{l}(l_{2}(\mathcal{A})) $ to be the class of operators in $B(l_{2}(\mathcal{A}))  $ that have $  {\mathcal{M} \Phi }_{+}$-decomposition defined above, but are not necessarily adjointable. Hence $ {\mathcal{M} \Phi }_{+} (H_{\mathcal{A}})= \widehat{ {\mathcal{M}\Phi}}_{l} (l_{2}(\mathcal{A})) \cap B^{a} ( H_{\mathcal{A}} ).$ Similarly, we set $  \widehat{ {\mathcal{M}\Phi}}_{r} (l_{2}(\mathcal{A})) $ to be the set of all operators in $B(l_{2}(\mathcal{A}))  $ that have $ {\mathcal{M} \Phi }_{-}$- decomposition but are not necessarily adjointable. Thus ${\mathcal{M} \Phi }_{-} (H_{\mathcal{A}})=\widehat{ {\mathcal{M}\Phi}}_{r} (l_{2}(\mathcal{A})) \cap B^{a} ( H_{\mathcal{A}} ).$ Finally, we set $\widehat{ {\mathcal{M}\Phi}} (l_{2}(\mathcal{A})) $ to be the set of all $\mathcal{A}$-Fredholm operators on $l_{2}(\mathcal{A}) $ in the sense of \cite{IM} that are not necessarily adjointable.
	\begin{remark} \label{R01} %\textbf{R01} 
		\cite{I} Notice that if $M,N$ are two arbitrary Hilbert  $\mathrm{C}^{*}$-modules, the definition above could be generalized to the classes $\mathcal{M}\Phi_{+}(M,N)$ and $\mathcal{M}\Phi_{-}(M,N)$.\\
		Recall that by \cite[ Definition 2.7.8]{MT}, originally given in \cite{MF}, when $\mathrm{F} \in \mathcal{M}\Phi(H_{\mathcal{A}})     $ and 
		$$ H_{\mathcal{A}} = M_{1} \tilde \oplus {N_{1}}‎‎\stackrel{\mathrm{F}}{\longrightarrow} M_{2} \tilde \oplus N_{2}= H_{\mathcal{A}} $$
		is an $ \mathcal{M}\Phi    $ decomposition for $  \mathrm{F}   $, then the index of $\mathrm{F}$ takes values in $K(\mathcal{A})$ and is defined by index $ \mathrm{F}=[N_{1}]-[N_{2}] \in K(\mathcal{A})    $ where $[N_{1}]    $ and $ [N_{2}] $ denote  the isomorphism classes of $ N_{1}    $ and $ N_{2} $ respectively. „By \cite[ Definition 2.7.9]{MT}, the index is well defined. 
	\end{remark}
	\begin{remark} \label{R08} %\textbf{R08}
		By \cite[Proposition 4.3]{I2} it follows that $\widehat{{\mathcal{M}\Phi}}_{l} (l_{2}(\mathcal{A})), \widehat{{\mathcal{M}\Phi}}_{r} (l_{2}(\mathcal{A}))  $ are closed under multiplication as these sets coincide with the sets of all left-invertible and all right invertible elements in the Calkin algebra $B(l_{2}(\mathcal{A})) / K(l_{2}(\mathcal{A}))  ,$ respectively.	
	\end{remark}	
	\begin{lemma} \label{L08} %\textbf{L08}
		Let $ F,D \in B^{a}(H_{\mathcal{A}})  $ and suppose that $ Im F, Im D, Im DF$ are closed. Then there exist closed submodules $X, W, M^{\prime}$ s.t. $Im F=W \oplus (\ker D \cap Im F), $ $Im D=Im DF \oplus X, \ker D = M^{\prime} \oplus (\ker D \cap Im F).$ Moreover, $H_{\mathcal{A}}=W \tilde{ \oplus} S(X) \oplus \ker D,$  where $S={D_{\mid_{\ker D^{\bot}}}}^{-1}  .$
	\end{lemma}
	\begin{proof}
		Since $Im F, Im D, Im DF$ are closed, by \cite[Theorem 2.3.3]{MT} they are orthogonally complementable. Also $\ker D  $ is orthogonally complementable. The operator $D_{\mid_{Im F}}  $ can therefore be viewed as an adjointable operator from $Im F $ into $Im D.$ Since $Im D_{\mid_{Im F}} =Im DF  $ is closed, by \cite[Theorem 2.3.3]{MT} $Im D=ImDF \oplus X $ for some closed submodule $X.$ Also $Im F = W \oplus \ker D_{\mid_{Im F}} = W \oplus (\ker D \cap Im F) $ for some closed submodule $W.$ Hence $H_{\mathcal{A}}= W \oplus (\ker D \cap Im F) \oplus Im F^{\perp}.$ Therefore $\ker D=(\ker D \cap Im F) \oplus M^{\prime},  $  where  $M^{\prime}=\ker D \cap (W \oplus Im F^{\perp}).$ Now, $ D_{\mid_{W}} $ is an isomorphism onto $Im DF.$ But $D_{\mid_{W}}=DP_{\mid_{W}}  $ where $P$ denotes the orthogonal projection onto $\ ker D^{\perp}  .$ It follows then that $P_{\mid_{W}}  $ must be bounded below, hence $P(W)$ is closed in $\ker D^{\perp}.$ In addition $P(W)=S(Im DF)$ where $S={D_{\ker D^{\perp}}}^{-1}  $ is the operator from $Im D$ onto $\ker D^{\perp}.$ Since $Im D=Im DF \oplus X  $ and $S$ is an isomorphism, we have that $ \ker D^{\perp}=S(ImDF) \tilde{ \oplus} S(X) .$ Hence $H_{\mathcal{A}}=S(Im DF) \tilde{ \oplus} S(X) \oplus \ker D.$ But $P_{\mid_{W}}  $ is an isomorphism from $W$ onto $S(Im DF).$ Hence $H_{\mathcal{A}}=W \tilde{ \oplus} S(X) \oplus \ker D .$ 
	\end{proof}	
	\begin{lemma} \label{L10} %\textbf{L10} 
		Let $M, N$ be closed submodules of $H_{\mathcal{A}}  $ such that $M \subseteq N  $ and $H_{\mathcal{A}}= M \oplus M^{\perp}  .$ Then $N=M \oplus (N \cap M^{\perp}).$
	\end{lemma}
	\begin{proof}
		Since $H_{\mathcal{A}}= M \oplus M^{\perp} ,$ every $ z \in N  $ can be written as $z=x+y$ where $x \in M, y \in M^{\perp}  .$ Hence $ z-x \in N ,$ as $z \in N  $ and $x \in M \subseteq N   .$ Thus $y \in N \cap M^{\perp} .$
	\end{proof}
\section{Semi-Fredholm operators and closed range operators over $W^{*}$-algebras}
We start with the following proposition.
\begin{proposition} \label{P01} %\textbf{P01} 
	Let $F \in  \widehat{{\mathcal{M}\Phi}}_{l} (l_{2}{(\mathcal{A})})$ or $F \in  {\mathcal{M}\Phi}_{+} (H_{\mathcal{A}}) .$ Then there exists a decomposition.
	$$H_{\mathcal{A}} = M_{0} \tilde \oplus M_{1}^{\prime} \tilde \oplus \ker  F‎\stackrel{F}{\longrightarrow}  N_{0} \tilde \oplus N_{1}^{\prime} \tilde \oplus {N_{1}^{\prime}}^{\prime}= H_{\mathcal{A}} $$
	w.r.t. which F has the matrix
	\begin{center}
	$\left\lbrack
	\begin{array}{lll}
		F_{0} & 0 & 0\\
		0 & F_{1} & 0 \\
		0 & 0 & 0\\
		\end{array}
		\right \rbrack
	$
	\end{center} 
	where $F_{0}$ is an isomorphism, $M_{1}^{\prime}$ and $\ker  F$ are finitely generated. Moreover $M_{1}^{\prime} \cong N_{1}^{\prime}$
	If $ F \in \widehat{ \mathcal{M} \Phi }_{l}  (l_{2}(\mathcal{A}))$ and $Im F$ is closed, then $Im F$ is complementable in $l_{2}(\mathcal{A}).$
\end{proposition}
\begin{proof}
	The first statement follows by the same arguments as in the proof of \cite[Proposition 3.6.8]{MT}. The second statement follows from the decomposition for $F$ from the first statement.
\end{proof}	
\begin{proposition} \label{P04} %\textbf{P04}
	If $D \in \widehat{ \mathcal{M} \Phi }_{r}  (l_{2}(\mathcal{A}))   $ and $Im D$ is closed and complementable in $l_{2}(\mathcal{A})  ,$ then the decomposition given above exists for the operator $D.$
\end{proposition} 
\begin{proof}
	Suppose that $F \in  \widehat{{\mathcal{M}\Phi}} _{r}(l_{2}(\mathcal{A}))$, let $l_{2}(\mathcal{A})=M_{1} \tilde \oplus {N_{1}}‎‎\stackrel{F}{\longrightarrow} M_{2} \tilde \oplus N_{2}=l_{2}(\mathcal{A})  ,$ an $\widehat{{\mathcal{M}\Phi}} _{r}(l_{2}(\mathcal{A}))$ decomposition for $F$, so that $N_{2}$ is finitely generated. Since $Im F$ is closed by assumption and $F(M_{1})=M_{2},F(N_{1})\subseteq N_{2},$ it follows easily that $F(N_{1})$ is closed. As $Im F$ is complementable by assumption, it follows that $F(N_{1})$ is complementable in $N_{2}$. Therefore $F(N_{1})$ is finitely generated projective, being a direct summand in a finitely generated, projective module $N_{2}$. Since $F_{\mid_{N_{1}}}:N_{1} \rightarrow F(N_{1})$ is an epimorphism, there exists a decimposition $N_{1}=N_{1}^{\prime} \tilde{ \oplus} \ker F$ where $N_{1}^{\prime} \cong F(N_{1}).$
\end{proof}
\begin{corollary} \label{C04} %\textbf{C04}
	1.) If $F \in  {\mathcal{M}\Phi} _{+}(H_{\mathcal{A}}) \setminus  {\mathcal{M}\Phi} (H_{\mathcal{A}}),$ then there exist $\epsilon >0$  such that  if $ D \in B^{a}(H_{\mathcal{A}})$ and  $\parallel D  \parallel < \epsilon,$ then $(F+D)$ is in  $ {\mathcal{M}\Phi} _{+}(H_{\mathcal{A}}) \setminus  {\mathcal{M}\Phi} (H_{\mathcal{A}})$ and $Im(F+D)^{\bot}$   is \underline{not} finitely generated. If $F \in \widehat{ {\mathcal{M}\Phi}}_{l} (l_{2}(\mathcal{A})) \setminus  \widehat{ {\mathcal{M}\Phi}} (l_{2}(\mathcal{A})) ,$ then the complement of $ \overline{Im F} $ is not finitely generated. \\
	2.)	If $F \in  {\mathcal{M}\Phi} _{-}(H_{\mathcal{A}}) \setminus  {\mathcal{M}\Phi} (H_{\mathcal{A}}),$ then there exists $\epsilon >0$  such that if $ D \in B^{a}(H_{\mathcal{A}}),$ and $\parallel D  \parallel < \epsilon,$ then $(F+D)\in  {\mathcal{M}\Phi} _{-}(H_{\mathcal{A}}) \setminus  {\mathcal{M}\Phi} (H_{\mathcal{A}})$ and $\ker (F+D)$ is \underline{not} finitely generated.
\end{corollary}
\begin{proof}
	It was shown in \cite[Theorem 4.1]{I}  that there exists an $\epsilon >0$  such that  $$(F+D)\in  {\mathcal{M}\Phi} _{+}(H_{\mathcal{A}}) \setminus  {\mathcal{M}\Phi} (H_{\mathcal{A}}),$$  whenever $\parallel D  \parallel < \epsilon.$ \\
	Now, since $$(F+D)\in  {\mathcal{M}\Phi} _{+}(H_{\mathcal{A}}) \setminus  {\mathcal{M}\Phi} (H_{\mathcal{A}}),$$ by Proposition \ref{P01} there exists a decomposition 
	$$H_{\mathcal{A}} = M_{1} \tilde \oplus (N_{1}^{\prime}   \oplus \ker (F+D) $$
	\begin{center}
		$\downarrow F+D$
	\end{center}
	$$H_{\mathcal{A}} = M_{2}  \oplus \overline{ (F+D)(N_{1}^{\prime})} \oplus Im(F+D)^{\bot}$$ w.r.t. which $(F+D)$ has the matrix
	\begin{center}
		$\left\lbrack
		\begin{array}{lll}
		(F+D)_{1} & 0 & 0\\
		0 & (F+D)_{4} & 0\\
		0 & 0 & 0\\
		\end{array}
		\right \rbrack
		,$
	\end{center} 
	where $(F+D)_{1}$ is an isomorphism and $N_{1}^{\prime}   \oplus \ker (F+D) $ is finitely generated, but $$\overline{ (F+D)(N_{1}^{\prime})} \oplus Im(F+D)^{\bot}$$ is \underline{not} finitely generated, as $(F+D) \notin  {\mathcal{M}\Phi} (H_{\mathcal{A}}).$ Now, since by Proposition \ref{P01} $\overline{ (F+D)(N_{1}^{\prime})} \cong N_{1}^{\prime}$ and $N_{1}^{\prime}$ is finitely generated being direct summand in a finitely generated submodule $N_{1}^{\prime} \oplus \ker (F+D),$  it follows that $Im(F+D)^{\bot}$ can \underline{not} be finitely generated, as 
	$$\overline{ (F+D)(N_{1}^{\prime})} \oplus Im(F+D)^{\bot}$$  is\underline{ not} finitely generated. The proof is similar in the case when $$F \in  \widehat{{\mathcal{M}\Phi}} _{l}(l_{2}(\mathcal{A}))  \setminus {\widehat{\mathcal{M}}\Phi}(l_{2}(\mathcal{A})),$$
	We just observe that the proof of \cite[Theorem 4.1]{I} does not require the adjointability of $F$ and moreover, Proposition \ref{P01} also applies in the case when $F \in  \widehat{{\mathcal{M}\Phi}} _{l}(l_{2}(\mathcal{A})) .$\\
	2)	This can be proved by passing to the adjoints and using \cite[Corollary 2.11]{I}.
\end{proof}
\begin{lemma} \label{L01} %\textbf{L01}
	If $F \in  \widehat{{\mathcal{M}\Phi}} _{r}(l_{2}(\mathcal{A}))  \setminus {\widehat{\mathcal{M}}\Phi}(l_{2}(\mathcal{A})),$ $ImF$ is closed and complementable, then the complement of $Im F$ is not finitely generated.
\end{lemma}
\begin{theorem} \label{T01} %\textbf{T01} 
	Let $F \in B^{a}(H_{\mathcal{A}}).$ Then $F \in  {\mathcal{M}\Phi} _{+}(H_{\mathcal{A}}) $ if and only if $\ker  (F-K)$ is finitely generated for all $K \in K^{*}(H_{\mathcal{A}}).$
\end{theorem}
\begin{proof}
	If $F \notin  {\mathcal{M}\Phi} _{+}(H_{\mathcal{A}}),$ choose a sequence $\lbrace x_{k} \rbrace \subseteq H_{\mathcal{A}}$ and an increasing sequence $\lbrace n_{k} \rbrace \subseteq \mathbb{N}$ s.t.
	$$x_{k} \in L_{{n}_{k}} \setminus L_{{n}_{k-1}} \text{ for all } k \in \mathbb{N},  \parallel x_{k} \parallel \leq 1 \text{ for all } k \in \mathbb{N} $$
	and
	$$\parallel Fx_{k} \parallel \leq 2^{1-2k} \text{ for all } k \in \mathbb{N} .$$
	By \cite[Lemma 3.2]{I} such sequence exists. Set
	$$K_{n}x=\sum_{k=1}^n \langle x_{k},x \rangle Fx_{k} \textrm{ for } x \in  H_{\mathcal{A}}.$$ 
	Then $K_{n} \in K^{*}(H_{\mathcal{A}}), $ for all $n.$ For $n > m,$ we have
	$$\parallel (K_{n}-K_{m})x \parallel \leq \sum_{k=m+1}^n \parallel x_{k} \parallel \parallel x \parallel \parallel Fx_{k} \parallel \leq \parallel x \parallel  \sum_{k=m+1}^n 2^{1-2(k+1)} ,$$
	so  $K_{n}-K_{m} \textrm{ vanishes as } n,m\longrightarrow \infty .$\\
	
	Let $K \in K^{*}(H_{\mathcal{A}}) $ be the limit of $ K_{n}^{\prime}s$ in the operator norm. Clearly, then $$ Kx=\sum_{k=1}^\infty \langle x_{k},x \rangle Fx_{k} \textrm{ ,  }  \forall x \in  H_{\mathcal{A}}.$$
	Observe next that by the construction of the sequence $ \lbrace x_{k} \rbrace ,$
	$$\langle x_{j},x_{k} \rangle = \delta _{j,k} \textrm{ , } \forall {j,k} \textrm{ as } x_{k}=L_{{n}_{k}} \setminus L_{{n}_{k-1}} \textrm{ , } \forall {k} $$ and the sequence $ \lbrace n_{n} \rbrace_{k} \subseteq \mathbb{N}$ is increasing.  Thus  $\lbrace  x_{k} \rbrace \subseteq \ker (F-K).$ Now, if $\ker (F-K)$ was finitely generated, then by \cite[Lemma 2.3.7]{MT} $\ker (F-K)$ would be an orthogonal direct summand in $H_{\mathcal{A}}.$ Hence, by the proof of  \cite[Theorem 2.7.5]{MT}, there exists an $n \in  \mathbb{N}$ such that ${p_n}_{{\mid}_{\ker (F-K)}}$ is an isomorphism from $\ker (F-K)$ onto some direct summand in $L_{n}$   (where $p_{n}$ is the orthogonal projection onto $L_{n}$ along $L_{n}^{\bot}).$ 
	In particular ${p_n}_{{\mid}_{\ker (F-K)}}$ is injective. However since, the sequence  $ \lbrace n_{k} \rbrace_{k} $ is increasing, we can find an $ n_{k_{0}} $  such that $n_{k}  \geq n  $ for all $ k  \geq k_{0} .$  Now, by construction, $x_{k} \in L_{{n}_{k}} \setminus L_{{n}_{k-1}}$ for all $k,$ so $x_{k} \in L_{n}^{\bot}$ for all ${k \geq k_{0}},$ as $n_{k}>n $ for all $ k > k_{0} .$ Consequently $p_n(x_{k})=0$ for all ${k \geq k_{0}}.$ As $ \lbrace x_{k} \rbrace_{k \geq k_{0}} \subseteq \ker (F-K),$ we get that $p_{n}$ is not injective, which is a contradiction. Thus we must have that  $\ker (F-K)$ is not finitely generated. On the other hand, if $ F \in  {\mathcal{M}\Phi} _{+}(H_{\mathcal{A}}),$ then 
	$$(F+K) \in  {\mathcal{M}\Phi} _{+}(H_{\mathcal{A}}) \textrm{ , } \forall K \in K^{*}(H_{\mathcal{A}}).$$ Now, as $\mathcal{A}$ is a $W^{*}$-algebra by assumption, then $\ker  (F-K)$ must be finitely generated for all $ K \in K(H_{\mathcal{A}}),$ as 
	$$(F-K) \in  {\mathcal{M}\Phi} _{+}(H_{\mathcal{A}}) \mbox{ for all } K \in K^{*}(H_{\mathcal{A}}),$$ 
	which holds by the same arguments as in the proof of \cite[Lemma 2.7.13]{MT}. This follows from the Proposition \ref{P01}.
\end{proof}
\begin{corollary}  \label{C03} %\textbf{C03} 
	Let $\mathcal{A}$ be a  $W^{*}$-algebra and $ F \in B^{a}(H_{\mathcal{A}}).$ Then $ F \in  {\mathcal{M}\Phi} _{-}(H_{\mathcal{A}})$ if and only if $Im(F-K^{*})^{\bot}$ is finitely generated for all $ K^{*} \in K^{*}(H_{\mathcal{A}}).$
\end{corollary}
\begin{proof}
	Suppose that $ F \notin  {\mathcal{M}\Phi} _{-}(H_{\mathcal{A}}).$ By \cite[Corollary 2.11]{I}, then $ F^{*} \notin  {\mathcal{M}\Phi} _{+}(H_{\mathcal{A}}).$ Hence there exists some $ K^{*} \in K^{*}(H_{\mathcal{A}})$ such that $\ker (F^{*}-K^{*})$ is not finitely generated. But $\ker (F^{*}-K^{*})=Im(F-K^{*})^{\bot}.$ On the other hand, if $ F \in  {\mathcal{M}\Phi} _{-}(H_{\mathcal{A}}),$ then by \\
	\cite[Corolary 2.11]{I}    $ F^{*} \in  {\mathcal{M}\Phi} _{+}(H_{\mathcal{A}}).$ Hence, by the Theorem \ref{T01} $\ker (F^{*}-K^{*})$ is finitely generated  for all $K^{*} \in K^{*}(H_{\mathcal{A}}),$  so $Im(F-K^{*})^{\bot}$ is finitely generated for all\\
	$ K^{*} \in K^{*}(H_{\mathcal{A}}).$
\end{proof}
%The operator in  $ {\mathcal{M}\Phi} _{+}$ with closed image have some special properties and we are going to  investigate them here.
%Dodatak RS4
Next we define another class of operators on $l_{2}(\mathcal{A}) .$
\begin{definition} \label{D03} %\textbf{D03} 
	Let $ F \in B(l_{2}(\mathcal{A})) .$ We say that $ F \in \widehat{{\mathcal{M}\Phi}}_{+}(l_{2}(\mathcal{A}))  $ if there exist a closed submodule $M$ and a finitely generated submodule $N$ s.t. $l_{2}(\mathcal{A})=M \tilde{\oplus} N  $ and $F_{\mid_{M}}  $ is bounded below.
\end{definition}
Note that we do not assume that $F(M)$ is complementable in $l_{2}(\mathcal{A})  .$ \\
Thus $\widehat{{\mathcal{M}\Phi}}_{l}(l_{2}(\mathcal{A})) \subseteq \widehat{{\mathcal{M}\Phi}}_{+} (l_{2}(\mathcal{A})),$ but the equality does not necessarily hold.
\begin{lemma} \label{L05} %\textbf{L05} 
	Let $F \in  B(l_{2}(\mathcal{A})) .$ Then $F \in \widehat{{\mathcal{M}\Phi}}_{+}(l_{2}(\mathcal{A})) $ iff $\ker (F-K)  $ is finitely generated for all $K^{*} (H_\mathcal{A}) .$
\end{lemma}
\begin{proof}
	Let $l_{2}(\mathcal{A})  $ be an $\widehat{{\mathcal{M}\Phi}}_{+} $ decomposition for $F$ from the definition above. Since $N$ is finitely generated, we may choose an $ n \in N $ s.t. $l_{2}(\mathcal{A})=L_{n}^{\perp} \tilde{\oplus} P  \tilde{\oplus} N   $ for some finitely generated, closed submodule $P.$ On $ L_{n}^{\perp} \tilde{\oplus} P $ $F$ is bounded below, so there exists some $ C>0  $ s.t. $\| Fx \| \geq C\|x\| $ for all $x \in L_{n}^{\perp} \oplus P .$ Next, if $K \in K^{*} (H_{\mathcal{A}})  ,$ by Proposition 2.1.1 \cite{MT} there exists some $m \geq n  $ s.t. $ \| K_{\mid_{L_{m}^{\perp}}} \| < C .$ Then $F-K$ is bounded below on $ L_{m}^{\perp} .$ Conversely, if $ F \notin \widehat{{\mathcal{M}\Phi}}_{+}(l_{2}(\mathcal{A})) ,$ then, in particular, $F$ is not bounded below on $L_{n}^{\perp}  $ for all $n  .$ We may hence repeat the construction from \cite[Lemma 3.2]{I} to get the sequence $\lbrace x_{k} \rbrace_{k}  $ s.t. the proof of Theorem \ref{T01} applies. The operator $K$ from this proof is adjointable and compact being the limit in operator norm of operators from $K^{*} (H_{\mathcal{A}})  .$ Moreover, $\ker (F-K)  $ is not finitely generated. 
\end{proof}
Set $\widehat{{\mathcal{M}\Phi}}_{-}(l_{2}(\mathcal{A}))=\lbrace G \in B(l_{2}(\mathcal{A}) \mid  $ there exists closed submodules $M,N,M^{\prime}  $ of $l_{2}(\mathcal{A})  $ s.t. $l_{2}(\mathcal{A})=M \tilde{\oplus} N,$ $N$ is finitely generated and $G_{\mid_{M^{\prime}}}  ,$ is an isomorphism onto $M \rbrace  .$
\begin{remark} \label{R09} %\textbf{R09} 
	We do not require that $M^{\prime}  $ is complementable in $l_{2}(\mathcal{A})  .$ Hence we have only the inclusion $\widehat{{\mathcal{M}\Phi}}_{r}(l_{2}(\mathcal{A})) \subseteq \widehat{{\mathcal{M}\Phi}}_{-}(l_{2}(\mathcal{A}))  ,$ but not necessarily the equality.
\end{remark}
\begin{proposition} \label{P10} %\textbf{P10} 
	Let $G \in \widehat{{\mathcal{M}\Phi}}_{-}(l_{2}(\mathcal{A}))  .$ 
	Then for every $ K \in K(l_{2}(\mathcal{A}))$ there exists an inner product equivalent to the initial one such that the orthogonal complement of $\overline{Im (G+K)}$ w.r.t this new inner product is finitely generated.
\end{proposition}
\begin{proof}
	Let $l_{2}(\mathcal{A})= M \tilde{\oplus} N  $ be an $ \widehat{{\mathcal{M}\Phi}}_{-} $ decomposition for $G$ from the definition above, let $M^{\prime} \subseteq l_{2}(\mathcal{A}) $ be s.t. $ G_{\mid_{M^{\prime}}} $ is an isomorphism onto $M$. Since $N$ is finitely generated, there exists an $ n \in N $ s.t. $ l_{2}(\mathcal{A})=L_{n}^{\perp} \oplus P \oplus N $ for some finitely generated submodule $P.$ If we let $\sqcap  $ denote the projection onto $L_{n}^{\perp} \oplus P  $ along $N,$ it follows that $ \sqcap_{\mid_{M}} $ is an isomorphism onto $L_{n}^{\perp} \oplus P  .$ Hence $\sqcap G_{\mid{M^{\prime}}}  ,$ is an isomorphism from $ M^{\prime}  $ onto $L_{n}^{\perp} \oplus P .$ If $K \in K (l_{2}(\mathcal{A}) ) ,$ then there exist an $ m \geq n $ s.t. $  \| q_{m}K \| < \| (\sqcap G_{\mid{M^{\prime}}})^{-1} \|^{-1} .$ Let $ M^{\prime \prime} =(\sqcap G_{\mid{M^{\prime}}})^{-1} (L_{m}^{\perp}).$ Then $\sqcap G_{\mid{M^{\prime \prime}}}=q_{m}G_{\mid{M^{\prime \prime}}}  $ and moreover $q_{m}(G-K)_{\mid{M^{\prime \prime}}} $ is an isomorphism onto $ L_{m}^{\perp}  .$ Now, $M^{\prime}= M^{\prime \prime} \tilde{\oplus} N^{\prime \prime  } $	where $N^{\prime \prime}= (\sqcap G_{\mid{M^{\prime}}})^{-1} (P \oplus (L_{m} \setminus L_{n})) .$ W.r.t. the decomposition 
	$$M^{\prime}=M^{\prime \prime} \tilde{\oplus} N^{\prime \prime} \stackrel{G-K}{\longrightarrow} L_{m}^{\perp} \oplus L_{m} = l_{2}(\mathcal{A}),$$ 
	$ G-K$ has the matrix
	$\left\lbrack
	\begin{array}{ll}
	(G-K)_{1} & (G-K)_{2} \\
	(G-K)_{3} & (G-K)_{4} \\
	\end{array}
	\right \rbrack,
	$  
	where $(G-K)_{1}=q_{m}(G-K)_{\mid_{M^{\prime \prime}}}  $ is an isomorphism. Hence, by the same arguments as in the proof of \cite[Lemma 2.7.10]{MT} there exists an isomorphism $U:M^{\prime} \longrightarrow M^{\prime } $ s.t. $V:l_{2}(\mathcal{A}) \longrightarrow  l_{2}(\mathcal{A}) $ s.t. $ G-K $ has the matrix 
	$\left\lbrack
	\begin{array}{ll}
	\overbrace{(G-K)_{1}} & 0 \\
	0 & \overbrace{(G-K)_{4}} \\
	\end{array}
	\right \rbrack,
	$  
	w.r.t. the decomposition 
	$$M^{\prime}=U(M^{\prime \prime}) \tilde{\oplus} U(N^{\prime \prime}) \stackrel{G-K}{\longrightarrow} V^{m} (L_{m}^{\perp}) \tilde{\oplus} V (L_{m})=l_{2}(\mathcal{A}) ,$$ where $\overbrace{(G-K)_{1}}  $ is an isomorphism. Moreover $V$ is s.t. $V (L_{m})=L_{m}  $ by construction  from the proof of \cite[Lemma 2.7.10]{MT}. Since $V (L_{m}^{\perp}) \subseteq  Im(G-K) \subseteq \overline{Im(G-K)}$ and $l_{2}(\mathcal{A})= V (L_{m}^{\perp}) \tilde{\oplus} L_{m} ,$ we obtain that $\overline{Im(G-K)}=V (L_{m}^{\perp})  \tilde{\oplus} (L_{m} \cap \overline{Im(G-K)})   .$ On $l_{2}(\mathcal{A})  $ we may replace the inner product by an equivalent one, so that $ V (L_{m}^{\perp})  $ and $L_{m}  $ form an orthogonal direct sum with respect to this new inner product. Since $L_{m}  $ is finitely generated and $L_{m} \cap \overline{Im(G-K)}  $ is a closed submodule of $L_{m}   ,$ by \cite[Lemma 3.6.1]{MT} we obtain that $$L_{m}={(L_{m} \cap \overline{Im(G-K)})^{\perp}}^{\perp}  \oplus (L_{m} \cap \overline{Im(G-K)})^{\perp} .$$ It follows then that $(L_{m} \cap \overline{Im(G-K)})^{\perp}  $ is finitely generated. Since $\overline{Im(G-K)}= V (L_{m}^{\perp}) \oplus  (L_{m} \cap \overline{Im(G-K)}),$ we have that ${\overline{Im(G-K)}}^{\perp}  $ is finitely generated.   Here, of course, the orthogonal complement is given w.r.t. the new inner product. 
\end{proof}
Note that from the proof of Proposition \ref{P10} it follows that if $\overline{Im(F-K)}  $ is complementable in $l_{2}(\mathcal{A}) $ (for $F \in  \widehat{{\mathcal{M}\Phi}}_{-} (l_{2}(\mathcal{A})), K \in K(l_{2}(\mathcal{A}))$), then the complement must be finitely generated.
\begin{lemma} \label{L04} %\textbf{L04} 
	Let $D \in B^{a} (H_{\mathcal{A}}) .$ Then $ D \in  {\mathcal{M}\Phi}_{-}(H_{\mathcal{A}}) $ iff there exist closed submodules $M,N$ such that $H_{\mathcal{A}}= M \tilde{\oplus} N,N  $ is finitely generated and $M \subseteq Im D.$	
\end{lemma}
\begin{proof}
	If $D \in B^{a} (H_{\mathcal{A}})  ,$ then such modules clearly exist from ${\mathcal{M}\Phi}_{-} $ decomposition of $D.$ Conversely, if such modules exist for $D \in B^{a} (H_{\mathcal{A}})  ,$ then $N$ is orthogonally complementable in $H_{\mathcal{A}}  $ by \cite[Lemma 2.3.7] {MT}. If $P  $ denotes the orthogonal projection onto $N^{\perp}  ,$ then $ P_{\mid_{M}} $ is an isomorphism onto $ N^{\perp} ,$ as $ M \tilde{\oplus} N= H_{\mathcal{A}} .$ Hence the operator $PD$ is an adjointable operator and $Im PD=N^{\perp}  .$ By  \cite[Theorem 2.3.3]{MT},  $\ker PD   $ is orthogonally complementable in $H_{\mathcal{A}}  .$ W.r.t. the decomposition 
	$$H_{\mathcal{A}}=\ker PD^{\perp} \oplus \ker PD ‎\stackrel{D}{\longrightarrow} N^{\perp} \oplus N =H_{\mathcal{A}},$$ $D$ has the matrix
	$\left\lbrack
	\begin{array}{ll}
	D_{1} & D_{2} \\
	D_{3} & D_{4} \\
	\end{array}
	\right \rbrack,
	$ 
	where $D_{1}=PD_{\mid_{\ker PD^{\perp}}}  $ is an isomorphism. Using the techniques of diagonalization as in the proof of \cite[Lemma 2.7.10]{MT} and the fact that $N$ is finitely generated, we obtain that $D \in  {\mathcal{M}\Phi}_{-}(H_{\mathcal{A}})   .$
\end{proof}
\begin{corollary} \label{C01} %\textbf{C01}
	$\widehat{{\mathcal{M}\Phi}}_{-}(l_{2}(\mathcal{A}))  \cap B^{a}(H_{\mathcal{A}})={\mathcal{M}\Phi}_{-}(H_{\mathcal{A}}) .$
\end{corollary}
% kraj dodatka S6 Obrisane Proposition 3.13 и Corollary 3.14. i dopisan nove 2 strane
\begin{lemma} \label{L11} %\textbf{L11}
	Let $ F,D \in B^{a}(H_{\mathcal{A}})   $ and suppose that $ Im F, Im D$ are closed. If $ ImF+ \ker D$ is closed, then $ImF+\ker D $ is orthogonally  complementable. 
\end{lemma}
\begin{proof}
	Suppose that $ImF+\ker D $ is closed. Since $Im F \oplus Im F^{\perp}
= H_{\mathcal{A}}$ by \cite[Theorem 2.3.3]{MT}, we have that $Im F + \ker D= Im F \oplus M^{\prime \prime}$, where $ M^{\prime \prime}=(ImF +\ker D) \cap Im F^{\perp},$ as $Im F \subseteq Im F + \ker D.$ Thus follows from Lemma \ref{L10}. Let $P$ denote the orthogonal projection  onto $Im F^{\perp}.$ Then  $ M^{\prime \prime}=P(ImF +\ker D)=P(Im F)+P(\ker D)=P(\ker D).$ This $Im(P_{\mid_{\ker D}})=M^{\prime \prime}.$ Now, since $Im D$ is closed, again by \cite[Theorem 2.3.3]{MT}, $\ker D$ is orthogonally  complementable in  $H_{\mathcal{A}}.$ Hence $P_{\mid_{\ker D}}$ is an adjointable operator from $\ker D$ into $Im F^{\perp}$ and its image is closed. Applying once again \cite[Theorem 2.3.3]{MT} to the operator $P_{\mid_{\ker D}},$ we obtain that $Im F^{\perp} = M^{\prime \prime} \oplus N^{\prime \prime}, \ker D= \ker (P_{\mid_{\ker D}}) \oplus M^{\prime} = (\ker D \cap Im F) \oplus M^{\prime}$ for some closed submodules $N^{\prime \prime},M^{\prime}.$ Then $P_{\mid_{M^{\prime}}}$ is an isomorphism onto $M^{\prime \prime}.$ It follows then that $ H_{\mathcal{A}}= (Im F \oplus N^{\prime \prime}) \tilde{\oplus} M^{\prime}.$ Moreover, since $H_{\mathcal{A}}=(\ker D \cap Im F) \oplus M^{\prime} \oplus \ker D^{\perp} ,$ we have that $\ker D \cap Im F $ is orthogonally complementable in $H_{\mathcal{A}}.$ Hence $ImF= (\ker D \cap Im F) \oplus M,$ where $M=Im F \cap (\ker D \cap Im F)^{\perp}.$ Here again we apply Lemma \ref{L10}. We obtain then that $ H_{\mathcal{A}}=((\ker D \cap Im F) \oplus M \oplus N^{\prime \prime} ) \tilde{\oplus} M^{\prime}= 
((\ker D \cap Im F) \tilde{ \oplus}  M^{\prime} \tilde{\oplus} M) \tilde{ \oplus} N^{\prime \prime}   =(\ker D + Im F) \tilde{ \oplus} N^{\prime \prime}.$ Let $ Q=(P_{\mid_{M^{\prime}}})^{-1}.$ Then $Q$ is a bounded adjointable operator from $M^{\prime \prime}$ onto $M^{\prime}.$ Consider now the operator $ \sqcap_{Im F} + J_{M^{\prime}}Q \sqcap_{M^{\prime \prime}}$ wher $\sqcap_{Im F}, \sqcap_{M^{\prime \prime}}$ denote the orthogonal projections onto $Im F,M^{\prime \prime} ,$ respectively and $J_{M^{\prime}}$ is the inclusion. Since $M^{\prime}$ is orthogonally complementable, $J_{M^{\prime}}$ is adjointable. Hence $\sqcap_{Im F}+ J_{M^{\prime}}Q \sqcap_{M^{\prime \prime}} \in B^{a}(H_{\mathcal{A}}).$ Moreover, $Im(\sqcap_{Im F}+ J_{M^{\prime}}Q \sqcap_{M^{\prime \prime}})=Im F \tilde{ \oplus} M^{\prime}= Im F + \ker D,$ which is closed by assumption. From  \cite[Theorem 2.3.3]{MT}, $Im F +\ker D$ is orthogonally complementable.  
\end{proof}
\begin{corollary} \label{C05} %\textbf{C05}
	Let $F,D \in B^{a}(H_{\mathcal{A}}) $ and suppose that $Im F, Im D$ are closed. Then $Im DF$ is closed if and only if $Im F+ \ker D$ is orthogonally complementable.
\end{corollary}
\begin{proof}
	By \cite[Corollary 1]{N}, $Im DF$ is closed if and only if $Im F + \ker D$ is closed. Now use Lemma \ref{L11}. 
\end{proof}
\begin{remark}
	The statement of Corollary \ref{C05} was already proved in \cite{Sh}, however, we have given here another, shorter proof. 
\end{remark}
Recall the definition of the „Dixmier angle“ between two Hilbert $C^{*}$-modules, given in \cite{Sh}.
\begin{definition} \label{D04} %\textbf{D04}
	Given two closed submodules $M,N$ of $H_{\mathcal{A}},$ we set $$c_{0} (M,N)=\sup \lbrace \parallel <x,y> \parallel \mid x \in M, y \in N, \parallel x \parallel, \parallel y \parallel  \leq 1 \rbrace.$$ 
	We say then that the Dixmier angle between $M$ and $N $ is positive if $ c_{0} (M,N) < 1 .$
\end{definition}
\begin{lemma} \label{L07} %\textbf{L07}
	Let $M, N$ be two closed, orthogonally complementable submodules of $H_{\mathcal{A}}$ and suppose that $M \cap N=\lbrace 0 \rbrace.$ Then $M+N$ is closed if and only if the Dixmier angle between $M$ and $N$ is positive.
\end{lemma}
\begin{proof}
	Suppose that the Dixmier angle between $M$ and $N$ is positive. We wish first to show that in this case there exists some constatnt $C>0$ s.t. whenever $x \in M, y \in N$ satisfy $\parallel x+y \parallel \leq 1 ,$ then $ \parallel x \parallel \leq C .$ To show this, observe first that since $M$ is orthogonally complementable in $H_{\mathcal{A}},$ there exist some $y^{\prime} \in M,y^{\prime \prime} \in M^{\perp}   $ s.t. $y=y^{\prime}+y^{\prime \prime}$ for $y \in N .$ Now, let $ c_{0} (M,N)=\delta <1 .$ Then 
	$$\sup  \lbrace \parallel <y,z> \parallel \mid z \in M, \parallel z \parallel = 1 \rbrace = \parallel y^{\prime} \parallel \leq \parallel y \parallel \delta .$$ 
	It follows that 
	$$\parallel y^{\prime \prime}  \parallel =\parallel  y- y^{ \prime}  \parallel \geq \parallel y  \parallel-  \parallel  y^{ \prime}  \parallel \geq (1-\delta) \parallel y \parallel= \dfrac{1-\delta}{\delta} \delta \parallel y \parallel \geq  \dfrac{1-\delta}{\delta} \parallel y^{ \prime}  \parallel   .$$ Now observe that $ <x+y,x+y> = <x+y^{ \prime} ,x+y^{ \prime}> + <y^{\prime \prime},y^{\prime \prime}>.$ By taking supremum over all states on $\mathcal{A}  ,$ we obtain $ \parallel x+y  \parallel \geq \max \lbrace \parallel  x+y^{ \prime} \parallel, \parallel  y^{\prime \prime} \parallel \rbrace .$ Thus, if $\parallel x+y  \parallel  \leq 1,$ then $ \parallel  x+y^{ \prime} \parallel, \parallel  y^{\prime \prime} \parallel \leq 1  .$ But, if $ \parallel  y^{\prime \prime} \parallel \leq 1 ,$ then by the calculation above, we get that $\parallel  y^{\prime } \parallel \leq  \dfrac{\delta}{1-\delta} .$ If in addition $\parallel x+ y^{\prime } \parallel \leq 1 ,$ then $1 \geq \parallel x  \parallel - \parallel  y^{ \prime} \parallel  \geq  \parallel x  \parallel - \dfrac{\delta}{1-\delta}.$ Hence we get $\parallel x  \parallel \leq 1+ \dfrac{\delta}{1-\delta} ,$ so we may set $C=1+ \dfrac{\delta}{1-\delta}.$
	Assume now that $ \lbrace x_{n}+y_{n} \rbrace_{n} $ is a Cauchy sequence in $M+N$ (here $x_{n} \in M,y_{n} \in N  $ for all $n$). From the arguments above we have that $\lbrace x_{n} \rbrace_{n}  $ must be then a Cauchy sequence in $M.$ Since $M$ is closed $ x_{n }\rightarrow x $ for some $ x\in M  .$ But, then $ \lbrace y_{n} \rbrace_{n} $ must be also convergent, so $  y_{n }\rightarrow y $ for some $ y \in N $ since $N$ is closed . Hence $ x_{n}+y_{n}$ converges to $x+y \in M+N$ as $n \rightarrow \infty .$ Thus $M+N$ is closed. Conversely, suppose that $M+N$ is closed. Then $M$ and $N$ form a direct sum, as $M \cap N = \lbrace 0 \rbrace  .$ Since $M$ is orthogonally complementable in $ H_{\mathcal{A}} ,$ we have that $M \tilde{ \oplus} N= M \oplus M^{\prime} $ where $M^{\prime}=(M \tilde{ \oplus} N) \cap M^{\perp}  .$ Here once again we apply Lemma \ref{L10}. If we let $P$ denote the orthogonal projection onto $M^{\perp}  ,$ we get that $ P_{\mid_{N}} $ is an isomorphism onto$M^{\prime}  .$ Hence there exists a constant $c>0  $ s.t. $\parallel Py \parallel \geqslant c \parallel y \parallel  $ for all $y \in N  .$ Since $\parallel P \parallel=1  ,$ we must have $c \leq 1  .$ Then $ \parallel (I-P)y \parallel \leq
	\parallel y \parallel-\parallel Py \parallel \leq (1-c)\parallel y \parallel $ for all $ y \in N .$ Consequently, we get 
	$$c_{0}(M,N)=\sup \lbrace \parallel <x,y> \parallel \mid x \in M, y \in N, \parallel x \parallel,\parallel y \parallel \leq 1  \rbrace=$$
	$$=\sup \lbrace \parallel <x,(I-P)y>\parallel \mid x \in M, y \in N, \parallel x \parallel,\parallel y \parallel \leq 1  \rbrace  \leq $$
	$$\leq \sup \lbrace \parallel x \parallel  (1-c) \parallel y \parallel   \mid x \in M, y \in N, \parallel x \parallel,\parallel y \parallel \leq 1  \rbrace  \leq  (1-c) <1  $$
\end{proof}
\begin{corollary} \label{C02} %\textbf{C02}
	Let $ F,D \in B^{a} (H_{\mathcal{A}})$ and suppose that $Im F, Im D$ are closed. Set $M=Im F \cap (\ker D \cap Im F)^{\perp},$ $M^{\prime}=\ker D    \cap (\ker D \cap Im F)^{\perp}.$ Then $Im DF$ is closed if and only if the Dixmier angle betwen $M^{\prime}  $ and $Im F,$ or equivalently the Dixmier angle between $M$ and $\ker  D$ is positive.  
\end{corollary}
Next we introduce the following notation: For two closed submodules $N_{1}, N_{2}$ of $M$ we write $N_{1} \preceq N_{2}$ when $N_{1}$ is isomorphic to a closed submodule of $N_{2}.$ 
\begin{proposition} \label{P03} %\textbf{P03}
	Let $F,G \in  \widehat{{\mathcal{M}\Phi}} _{l}(l_{2}(\mathcal{A}))$ with closed images and suppose that $Im GF$ is closed. Then $Im F, Im G, Im GF $are complementable in $l_{2}(\mathcal{A})  .$ Moreover, if we let $Im F^{0}, Im G^{0}, Im GF^{0}  $ denote the complements of $Im F, Im G, Im GF,$ respectively, then
	$$ImGF^{0} \preceq ImF^{0} \oplus ImG^{0},$$
	$$\ker GF \preceq \ker G \oplus \ker F .$$
	If $ F,G \in \widehat{{\mathcal{M}\Phi}}_{r} (l_{2}(\mathcal{A}))  $ and $Im F, Im G, Im GF$ are closed, then the statement above holds under additional assumption that $Im F, Im G, Im GF$ are complementable in $l_{2}(\mathcal{A})  .$
\end{proposition}
\begin{proof}
	Since $F \in \widehat{{\mathcal{M}\Phi}}_{l}(l_{2}(\mathcal{A})),$ $F$ has the decomposition given in Proposition \ref{P01}. Then $N_{1}^{\prime}=\overline{F(M_{1}^{\prime})}$ where we use the notation from Proposition 3.1. Now, since $ImF$ is closed by assumption, $N_{1}^{\prime}=F(M_{1}^{\prime}).$ Hence $ImF=N_{0} \tilde{\oplus} N_{1}^{\prime}$ and so ${N^{\prime}}^{\prime}=ImF^{0}.$ Since $ImG,ImGF$ are closed, by the same arguments $ImG^{0}, ImGF^{0}$ exist, as $G,GF \in \widehat{{\mathcal{M}\Phi}}_{r}(l_{2}(\mathcal{A})).$ Here we use that  $FG \in \widehat{{\mathcal{M}\Phi}}_{r} (l_{2}(\mathcal{A}))  $ by Remark \ref{R08} as $F,G \in \widehat{{\mathcal{M}\Phi}}_{r}(l_{2}(\mathcal{A})).$ Now, since $\ker G$ is self-dual, being finitely generated and since $\ker G \cap ImF$ is the kernel of the projection onto $Im F^{0}$ along $ImF$ restricted to $\ker G,$ by \cite[Corollary 3.6.4]{MT} we may deduce that	$ \ker G= (\ker G \cap ImF) \oplus M^{\prime}$  
	for some closed submodule $M^{\prime}.$ Hence $l_{2}(\mathcal{A})=(\ker G \cap Im F) \tilde{ \oplus} M^{\prime} \tilde{ \oplus} \ker G^{0} ,$ so $ \ker G \cap Im F $ is complementable in $l_{2}(\mathcal{A})  .$ It follows by the similar arguments as in Lemma \ref{L10} that $\ker G \cap Im F  $ is complementable in $Im F,$ being a submodule of $Im F.$ Thus $Im F= (\ker G \cap Im F)\tilde{ \oplus}M,$ where $M$ is the intersection of $Im F$ and the complement of $ \ker G \cap Im F $ in $l_{2}(\mathcal{A}).$ Observe also that by Proposition \ref{P01}. $\ker G$  and $\ker F$ are complementable in $l_{2}(\mathcal{A}).$ Since $ Im G,ImGF$ are both complementable in $l_{2}(\mathcal{A})$ and $ ImGF \subseteq ImG, $ it follows that $ImG=ImGF \tilde{ \oplus} X$ where $X=ImG \cap ImGF^{0}$ by the similar or arguments as in Lemma \ref{L10}.  So we have $l_{2}(\mathcal{A})=(\ker G \cap ImF) \tilde{ \oplus} M^{\prime} \tilde{ \oplus} \ker G^{0}=(\ker G \cap ImF) \oplus M \tilde{ \oplus} ImF^{0}=ImGF \tilde{ \oplus} X \tilde{ \oplus} ImG^{0}=l_{2}(\mathcal{A}).$  (where $\ker G^{0}$ denotes complented, closed submodule to $\ker G$). Let $\sqcap \in B  (l_{2}(\mathcal{A})) $ denote the projection onto $\ker G^{0}  $ along $ \ker G .$ We have that $ G_{\mid_{M}} $ is an isomorphism onto $Im GF  $ and moreover $ G_{\mid_{M}} = G \sqcap_{\mid_{M}} .$ Since $G_{\mid_{\ker G^{0}}}  $ and $ G\sqcap_{\mid_{M}} $ are isomorphisms, it follows that $\sqcap_{\mid_{M}} $ is an isomorphism. Let $S=(G_{\mid_{\ker G^{0}}})^{-1}   .$ Then $ \sqcap (M)=S(ImGF) .$ As $ \ker G^{0}=S(ImGF) \tilde{ \oplus} S(X) ,$ it follows that $l_{2}(\mathcal{A}) = M \tilde{ \oplus} S(X) \tilde{ \oplus} \ker G = M \tilde{ \oplus} S(X) \tilde{ \oplus} M^{\prime} \tilde{ \oplus} (\ker G \cap Im F ) .$ But, we have also $Im F = (\ker G \cap Im F) \oplus M.$ It follows that $Im F^{0} \cong S(X) \tilde{ \oplus} M^{\prime} \cong X \oplus M^{\prime}.$
	% Kraj dodatka RS7
	But, from the expresion above we see that $ImGF^{0} \cong X \oplus ImG^{0} \preceq X \oplus M^{\prime} \oplus ImG^{0} \cong ImF^{0} \oplus Im G^{0}.$	
    In the case when $F,G \in \widehat{{\mathcal{M}\Phi}}_{r} (l_{2}(\mathcal{A}))   $ and $Im F, Im G, Im GF$ are closed, complementable in $l_{2}(\mathcal{A}) ,$ we may apply the same proof as above, but we only need to argue first why $\ker G \cap Im F  $ is complementable. This can be deduced in the following way: Since $ F \in \widehat{{\mathcal{M}\Phi}}_{r}(l_{2}(\mathcal{A})) ,$ and  $Im F$ is closed and complementable, by Proposition \ref{P04}  $\ker F  $ is complementable in $l_{2}(\mathcal{A})  .$ Hence $\ker F= \ker F \tilde{\oplus} W,$ where $W$ is the intersection of $\ker GF$ and the complement of $\ker F,$ which follows again by similar arguments as in Lemma \ref{L10}. We have that $ F_{\mid_{W}} $ is an isomorphism onto $ \ker G \cap Im F.$ Next, again since $GF \in \widehat{{\mathcal{M}\Phi}}_{r}(l_{2}(\mathcal{A}))    $ and $Im GF   $ is closed and complementable, we get that $l_{2}(\mathcal{A}) = \ker GF \tilde{\oplus} M  $ for some closed submodule $M.$ Hence $l_{2}(\mathcal{A}) = \ker F \tilde{\oplus} W \tilde{\oplus} M .$ On $W \tilde{\oplus} M  $ is an isomorphism onto $Im F  ,$ so $$Im F =F(W) \tilde{\oplus} F(M)= (\ker G \cap Im F) \tilde{\oplus} F(M) . $$
	Therefore 
	$$l_{2}(\mathcal{A})=(\ker G \cap Im F) \tilde{\oplus} F(M) \tilde{\oplus} Im F^{0}  $$
	where $Im F^{0}  $ denotes the complement of $Im F  .$ It follows that $\ker G \cap Im F  $ is complementable.
	%\end{remark}	
	In order to deduce that $\ker DF \preceq (\ker D \oplus \ker F) ,$ one can proceed in exactly the same way as in the proof of \cite[Theorem 1.2.4]{ZZRD} to obtain that $\ker DF=\ker F \tilde \oplus W$  where  $W \cong (\ker D \cap ImF).$ The rest follows.
\end{proof}
\begin{lemma} \label{L06} %\textbf{L06} 
	Let $ F,D \in B^{a}( H_{\mathcal{A}} )$ and suppose that $ Im F,Im D, Im DF $ are closed. Then 
	$$Im DF^{\perp} \preceq Im F^{\perp} \oplus Im D^{\perp}  $$
	$$\ker DF \preceq \ker D \oplus \ker F  $$
\end{lemma} 
\begin{proof}
	The statement can be proved in exactly the same way as in the Proposition \ref{P03} as $Im F, Im D, Im DF$ will be then orthogonally complementable in $H_{\mathcal{A}}$ by \cite[Theorem 2.3.3]{MT}. Again, we only need to argue that $\ker D \cap Im F  $ is orthogonally complementable in $H_{\mathcal{A}}  .$ Now $ D_{\mid_{Im F}} $ is an adjointable operator from $Im F$ into $H_{\mathcal{A}}   $ as $D \in B^{a} (H_{\mathcal{A}})  $ and $Im F$ is orthogonally complementable in $H_{\mathcal{A}}  .$ Moreover $Im D_{\mid_{Im F}}= Im DF,$ which is closed by assumption. From \cite[Theorem 2.3.3]{MT} it follows that $\ker D_{\mid_{Im F}}=\ker D \cap Im F  $  is orthogonally complementable in $Im F$. Thus $Im F = (\ker D \cap Im F) \oplus M  $ for some closed submodule $M.$ Hence $H_{\mathcal{A}} =(\ker D \cap Im F) \oplus M \oplus Im F^{\perp} .$
\end{proof}
\begin{lemma} \label{L09} %\textbf{L09} 
	Let $F,G \in \widehat{{\mathcal{M}\Phi}} (l_{2}(\mathcal{A}))    $ and suppose  that, $Im G, Im F$ are closed. Then $Im GF$ is closed if and only if $Im F+\ker G$ is closed and complementable.
\end{lemma}
\begin{proof}
	If $Im F + \ker G$ is closed, then $Im GF$ is closed by \cite[Corollary 1]{N}. Conversely, assume that $Im GF$ is closed. Now, by Remark \ref{R08} $GF \in \widehat{{\mathcal{M}\Phi}} (l_{2}(\mathcal{A}))   $ as $G, F$ are so. Then by Proposition \ref{P01} $Im GF$ is complementable. Moreover, $\ker G \cap Im F  $ is complementable in $Im F$ by the same arguments as earlier, because $F,G \in \widehat{{\mathcal{M}\Phi}} (l_{2}(\mathcal{A}))   .$ So we may write $Im F$ as $Im F=(\ker G \cap Im F) \tilde{ \oplus} \tilde{M}.$ We have then that $G_{\mid_{\tilde{M}}}  $ is an isomorphism onto $Im GF.$ Let $Im GF^{0}  ,Im G^{0}$ denote the complements of $Im GF, Im G,$ respectively. Then, since $Im GF \subseteq Im G,$ by the proof of Lemma \ref{L10} $Im G = Im GF \tilde{ \oplus} (Im G F^{0} \cap Im G)  .$ Hence, we get $l_{2}(\mathcal{A}) = Im GF \tilde{ \oplus} (Im GF^{0} \cap Im G) \tilde{ \oplus} Im G^{0}.$ Moreover, since $G \in \widehat{{\mathcal{M}\Phi}} (l_{2}(\mathcal{A}))  $ and $Im G$ is closed by assumption, by Proposition \ref{P01} $\ker G$ is complementable in $l_{2}(\mathcal{A})  .$ If we let $\ker G^{0}  $ denote the complement of $\ker G,$ we have then that $G_{\mid_{\ker G^{0}}}  $ is an isomorphism onto $Im G.$ Combining all these facts together, we are then in the position to apply the same arguments as in the proof of Lemma \ref{L08} to obtain that $ l_{2}(\mathcal{A})= \tilde{M} \tilde{ \oplus} S^{\prime}(Om GF^{0} \cap Im G) \tilde{ \oplus} \ker G  $ where $S^{\prime}=(G_{\mid_{\ker G^{0}}})^{-1}  .$ Hence $\tilde{M} \tilde{ \oplus} \ker G  $ is closed and complementable in $l_{2}(\mathcal{A}).$ But $ \tilde{M} \tilde{ \oplus} \ker G=\tilde{M} \tilde{ \oplus} (\ker G \cap Im F) \tilde{ \oplus} \tilde{M}^{\prime}=\ker G +Im F .$\\
\end{proof}
\begin{remark} \label{R04} %\textbf{R04} 
	By Sakai's theorem, since $\mathcal{A}$ is a $W^{*}$  algebra, $\mathcal{A}$ is a dual space of a certain Banach space, hence $\mathcal{A}$ can also be equipped with the $w^{*}$-topology. Consequently $\mathcal{A}^{\mathbb{N}}$ can be equipped with the product $w^{*}$-topology.  Since $H_{\mathcal{A}} \subseteq \mathcal{A}^{\mathbb{N}}, H_{\mathcal{A}}$ has a subspace topology inherited from the product $w^{*}$ -topology on $\mathcal{A}^{\mathbb{N}}.$
\end{remark}
The next lemma is motivated by the well known result \cite[Theorem 1.2.3]{ZZRD} in the classical semi-Fredholm theory on Hilbert spaces which states that if $H$ is a Hilbert space and $F \in B(H),$ then $F \in \Phi_{+}(H) $ if and only if for every bounded sequence $\lbrace x_{n} \rbrace $ in $H$ which does not have a convergent subsequence,  $\lbrace Fx_{n} \rbrace $ does not have a convergent subsequence.
\begin{lemma} \label{L02} %\textbf{L02} 
	Let $F \in  B^{a}(H_{\mathcal{A}}) $ and suppose that $ImF$ is closed, let $\lbrace x_{n} \rbrace $ be a sequence in $H_{\mathcal{A}} $ s.t. $ \lbrace P_{\ker  F}x_{n} \rbrace $  is a bounded sequence in $H_{\mathcal{A}}.$ If $\lbrace x_{n} \rbrace$ does not have a convergent  subsequence in the product $w^{*}$-topology, then $\lbrace F x_{n} \rbrace $ does not have a convergent subsequence in the norm topology of $H_{\mathcal{A}}.$
\end{lemma}
\begin{proof}
	Notice first that since $ImF$ is closed , then $\ker F$ is an orthogonal direct summnand in $H_{\mathcal{A}}$ by  \cite[Theorem 2.3.3]{MT} $F_{{\mid}_{\ker F^{\bot}}}$ is an isomorphism from $\ker F^{\bot}$ onto $ImF.$ Also, in the  statement of the theorem $P_{\ker F}$ denotes the orthogonal projection onto $\ker F$ along $\ker F^{\bot}.$ Therefore, $F_{{\mid}_{\ker F^{\bot}}}$ has a bounded inverse from $ImF$ onto $\ker  F^{\bot}.$ Now, since $H_{\mathcal{A}} = \ker  F \tilde \oplus \ker  F^{\bot}, x_{n}$ can be written as 
	$$x_{n}=u_{n}+v_{n},u_{n} \in \ker  F , v_{n} \in \ker  F^{\bot} \forall n .$$ 
	Suppose that $\lbrace Fx_{n} \rbrace$ has a convergent subsequence  $\lbrace Fx_{n_{k}} \rbrace _{k}.$ Then $$(F_{{\mid}_{(\ker F)^{\bot}}})^{-1} Fx_{n_{k}}=v_{n_{k}}$$ is a convergent subsequence, hence it is convergent in the product $w^{*}$- topology (since the sequence \\
	$\lbrace v_{{n}_{k}} \rbrace_{k} $ coordinatevise is convergent in the norm of $\mathcal{A},$ hence in the $w^{*}$- topology of $\mathcal{A}$). By assumption of the theorem, $\lbrace P_{\ker F}x_{n_{k}} \rbrace $  is bounded, hence since $P_{\ker F}x_{n_{k}}=u_{n_{k}}$ we get that $\lbrace u_{n_{k}} \rbrace \subseteq  (  \overline{B}_{N}^{*}(0))^{\mathbb{N}} $ where $\overline{B}_{N}^{*}(0)$ is the closed ball with center in 0 and radius $N$ in $\mathcal{A}$ and $N$ is chosen such that $\parallel u_{n_{k}} \parallel \leq N,$ for all $k.$ By Alaoglu theorem, $  \overline{B}_{N}^{*}(0) $ compact, hence by Tychonoff theorem, $ ( \overline{B}_{N}^{*}(0))^{\mathbb{N}} $ is compact in the product $w^{*}$-topology.\\
	Therefore, $\lbrace u_{n_{k}} \rbrace$ has a convergent subsequence in the product $w^{*}$-topology , say $\lbrace u_{n_{k_{j}}} \rbrace.$ Hence $x_{n_{k_{j}}}=u_{n_{k_{j}}}+v_{n_{k_{j}}}$ is a convergent subseqence in the product $w^{*}$ topology, which is not possible.
\end{proof}
\begin{remark} \label{R05} %\textbf{R05} 
	Observe that in Lemma \ref{L02} we do not assume that $F \in  {\mathcal{M}\Phi_{+}}(H_{\mathcal{A}})  ,$ but only that  $Im F$ is closed. However, we have only implication in this lemma and not the equivalence. The key argument in proving \cite[Theorem 1.2.3]{ZZRD} is that the unit ball in the finite dimensional space is compact. In our generalized situation we do not have this tool at disposition, however we have Alaoglu's theorem as a counterpart.   
\end{remark}
\begin{remark} \label{R07} %\textbf{R07}
	In Lemma \ref{L02}, if $F$ was not adjointable, then we would need to assume f. ex. that $F \in  \widehat{{\mathcal{M}\Phi}}_{l} (l_{2}(\mathcal{A})),$ because in the nonadjointable case we do not have in general that $ImF$ would be complmentable if it is closed. However, if $F \in  \widehat{{\mathcal{M}\Phi}}_{l}(l_{2}(\mathcal{A}))$ and $Im F$ is closed, then, by Proposition \ref{P01} we have that $ImF$ is complementable.
\end{remark}
\begin{lemma} \label{L03} %\textbf{L03} 
	Let $ F \in  {\mathcal{M}\Phi}  (M)$ s.t. $ImF$ is closed, where $M$ is a Hilbert $W^{*}$-module. Then there exists an $ \epsilon >0$ such that for every $ D \in B^{a}(M)$ with $ \parallel D \parallel < \epsilon$ we have 
	$$ \ker (F+D) \preceq \ker  F \textrm{ , } Im(F+D)^{\bot}\preceq ImF^{\bot}.$$ 
\end{lemma}
\begin{proof}
	Since $F \in  {\mathcal{M}\Phi} (M)$ has closed image, w.r.t the decomposition
	$$M=\ker F^{\bot} \tilde \oplus \ker F‎\stackrel{F}{\longrightarrow}  ImF \tilde \oplus ImF^{\bot}=M$$
	has the matrix
	\begin{center}
		$\left\lbrack
		\begin{array}{ll}
		F_{1} & 0 \\
		0 & 0 \\
		\end{array}
		\right \rbrack
		$
	\end{center}
	where $F_{1}$ is an isomorphism by \cite[Theorem 2.3.3]{MT} . By the proof of \cite[Lemma 2.7.10]{MT}, there exists an $\epsilon>0$ such that if $\parallel F- \tilde D \parallel < \epsilon$ for some  $ \tilde D \in B^{a}(M),$ then $ \tilde D$ has the matrix
	\begin{center}
		$\left\lbrack
		\begin{array}{ll}
		\tilde D_{1} & 0 \\
		0 & \tilde D_{4}  \\
		\end{array}
		\right \rbrack
		$
	\end{center}
	w.r.t. the decomposition
	$$M=U_{1}(\ker F^{\bot}) \tilde \oplus U_{1}(\ker F)_‎\stackrel{\tilde{D}}{\longrightarrow}  U_{2}^{-1}(ImF) \tilde \oplus U_{2}^{-1}(ImF^{\bot})=M$$
	where $U_{1},U_{2} $ and $\tilde D_{1}$ are isomorphisms. It follows then that $$\ker \tilde D \subseteq  U_{1}(\ker F)\cong \ker F.$$ Set $D=\tilde{D}-F,$ then $\tilde{D}=F+D.$ Hence $\ker (F+D)  \preceq \ker  F.$ Observe now that $U_{2}^{-1}(ImF)\subseteq Im \tilde{D}.$ Hence $$Im \tilde{D}^{\bot}\cap U_{2}^{-1}(ImF)= \lbrace 0 \rbrace,$$ so ${P_{U_{2}^{-1}(ImF^{\bot})}}_{{\mid}_{Im \tilde{D}^{\bot}}}$ is injective, where $P_{U_{2}^{-1}(ImF^{\bot})} $ denotes the projection onto \\
	$U_{2}^{-1}(ImF^{\bot})$ along $U_{2}^{-1}(ImF).$ Since $\tilde{D} \in  {\mathcal{M}\Phi} (M), Im \tilde{D}^{\bot}$  is finitely generated, hence self-dual. By \cite[Corollary 3.6.7]{MT}, it follows then that $Im \tilde{D}^{\bot}$ is isomorphic to a direct summand in $U_{2}^{-1}(ImF^{\bot}).$ Since  $U_{2}^{-1}(ImF^{\bot}) \cong ImF^{\bot}, $ it follows that $Im \tilde{D}^{\bot} \preccurlyeq ImF^{\bot}.$ 
\end{proof}
\begin{remark} \label{R10}
	Lemma \ref{L03} are also valid in the case when $F \in \widehat{ {\mathcal{M}\Phi}} (  M  )   $ with closed image because in this case, by Proposition \ref{P01}, there exists a decomposition 
	$$  M = \ker F^{0} \tilde{ \oplus} \ker F \stackrel{F}{\longrightarrow} Im F \tilde{ \oplus} Im F^{0} =M$$
	and $F_{\mid_{\ker F^{0}}}  $ is an isomorphism onto $Im F.$ By following the proof of Lemma \ref{L03} we obtain that $\ker(F+G) \preceq \ker F  $ and $Im (F+G)^{\perp} \preceq Im F^{0}  $ when $ G \in B(l_{2}(\mathcal{A})) $ is s.t. $\parallel G \parallel $ is sufficiently small. If $\overline{Im (F+G)}  $ is complementable, then $Im (F+G)^{0} \preceq Im F^{0}   $ (where $Im (F+G)^{0}   $ denotes the complement of $ \overline{Im (F+G)}  .)$
\end{remark} 
\begin{definition} \label{D02} %\textbf{D02} 
	Let $M$ be a countably generated Hilbert $W^{*}$- module. 
	For\\
	  $F \in  {\mathcal{M}\Phi} (M),$ we say that F satisfies the condition (*)  if the following holds:\\
	1) $ImF^{n}$ is closed for all $n$\\
	2) $F(\bigcap\limits_{n=1}^{\infty} Im(F^{n}))=\bigcap\limits_{n=1}^{\infty} Im(F^{n})$
\end{definition}
	If we have a sequence of decreasing complementable submodules $N_k^{\prime}s,$ then their intersection in general (for $C^{*}$-algebras) is not complementable, but it is complementable for $W^{*}$-algebras. This is true due to the possibility to define a $w^{*}$-(or weak) direct sum of submodules, as opposed to the standard $l_2$ construction. Let $N_{k-1}=N_k\oplus L_k$. Then we can define $w^{*}-\oplus_k L_k$ as the set of sequences $(x_k)$, $x_k\in L_k$, such that the sum $\sum\limits_{k=1}^{\infty}\langle x_k,x_k\rangle$ is convegent in $A$ with respect to the *-strong topology, as opposed to the norm topology. Then it is easy to see that $N_0=\bigcap\limits_{k=1}^\infty N_k\oplus (w^*-\oplus_k L_k).$\\
	Note that if $M$ is an ordinary Hilbert space, then (*) is always satisfied  for any\\
	$ F \in  {\mathcal{M}\Phi} (M)$ by \cite[Theorem 1.1.9]{ZZRD}. There are also other examples of Hilbert $W^{*}$-modules where the condition (*) is automatically satisfied for an $\mathcal{A}$-Fredholm operator $F$ as long as $F$ has closed image.
\begin{example} \label{E01} %\textbf{E01} 
	Let $\mathcal{A}$ be a commutative von Neumann  algebra with a cyclic vector, that is $\mathcal{A} \cong L^{\infty}(X,\mu)$ where $X$ is a compact topological space and $u$ is a Borel probability measure and consider $\mathcal{A}$ as a Hilbert module over itself. If $F$ is an $\mathcal{A}$-linear operator on $\mathcal{A},$ it is easily seen that $Im(F^{k})=Span_{\mathcal{A}} \lbrace (F(1))^{k} \rbrace $ for all $k.$
	Let $S=(F(1)^{-1}(\lbrace 0 \rbrace))^{c}.$ Then one can show that $ImF=ImF^{k}=Span_{\mathcal{A}}\lbrace \chi_{S} \rbrace$  for all $k$ if we assume that $F(1)$ is bounded away from 0 on $S,$ hence invertible on $S.$ But if $F$ is an $\mathcal{A}$-Fredholm  with closed image, then this is the case. Indeed, 
	$$\ker  F=   \lbrace f \in \mathcal{A} \mid f_{{\mid}_{S}}=0  \textrm{  }\mu\textrm{  } a.e.\textrm{  } on\textrm{  } S   \rbrace= Span_{\mathcal{A}}\lbrace \chi_{S^{c}} \rbrace,  \textrm{ so } \ker F^{\bot}=Span_{\mathcal{A}}\lbrace \chi_{S} \rbrace.$$
	Since $F$ is then bounded below on $\ker  F^{\bot},$ we have
	$$\parallel F(f) \parallel_{\infty}= \parallel  fF(1) \parallel_{\infty}\geq C \parallel f  \parallel_{\infty} $$ 
	for all $f$ being $0 $ $\mu$-almost everywhere on $S^{c}$ and for some constant $C>0.$ But, if  $$\mu   ((F(1)^{-1}(B(0, \frac{1}{n})))\cap S)>0 \textrm{  } \forall n,$$
	then letting 
	$$f_{n}=\chi_ {((F(1)^{-1}(B(0, \frac{1}{n})))\cap S)},$$ 
	we get $\parallel f_{n}  \parallel_\infty=1 \textrm{ , for all } n $ and 
	$$ F(f_{n})=f_{n}F(1)=\chi_ {((F(1)^{-1}(B(0, \frac{1}{n})))\cap S)}F(1).  $$
	It follows that $F$ will not be hounded below on $(\ker F)^{\bot}$ which is a contradiction. Observe now that
	$$Im(F)=Im(F^{k})=Span_{\mathcal{A}}\lbrace \chi_{s} \rbrace=(\ker F)^{\bot}  \forall k,$$
	so $F(Im^{\infty}(F))=Im^{\infty}(F) $ 	where $Im^{\infty}(F) $ denotes $\bigcap\limits_{k=1}^\infty Im(F^{k}).$\\
	Recall that for a $W^{*}$-algebra $\mathcal{A},$  $G(\mathcal{A})$ denotes the set of all invertible elements in $\mathcal{A}$ and $Z(\mathcal{A})=\lbrace \beta \in \mathcal{A} \mid \beta \alpha = \alpha \beta $ for all $\alpha \in \mathcal{A} \rbrace .$ We have then the following theorem.
\end{example}
\begin{theorem} \label{T02} %\textbf{T02} 
	Let $F \in  {\mathcal{M}\Phi}  (\tilde M)$  where $ \tilde M$ is countably generated  Hilbert $\mathcal{A}$-module and suppose that $F$ satisfies \textrm{(*)}. Then there exists an $ \epsilon > 0$ s.t. if $\alpha \in Z(\mathcal{A}) \cap G(\mathcal{A})$ and $\parallel \alpha \parallel < \epsilon ,$ then $[\ker (F-\alpha I)]+[N_{1}]=[\ker F]$ and\\
	$[Im(F-\alpha I)^{\bot}]+[N_{1}]=[Im (F)^{\bot}] $  for some fixed, finitely generated closed submodule $N_{1}.$
\end{theorem}
\begin{proof}
	Since $F \in  {\mathcal{M}\Phi} (\tilde M)$  has closed image, then by  Lemma \ref{L03}, there exists an $ \epsilon_{1}>0 $  such that  if $ \parallel \alpha  \parallel < \epsilon_{1} , \alpha \in Z(\mathcal{A}) \cap G(\mathcal{A}),$ then 
	$$\ker (F-\alpha I) \preceq \ker F, Im(F-\alpha I)^{\bot} \preceq ImF^{\bot} $$ 
	and by the proof of \cite[Lemma 2.7.10]{MT} $\textrm{index} (F-\alpha I)=\textrm{index}F.$ Now, by the same arguments as in the proof of \cite[Theorem 1.7.7]{ZZRD}, since $\alpha \in G(\mathcal{A}) \cap Z(\mathcal{A}) ,$ we have
	$$ \ker (F-\alpha I) \subseteq  Im^{\infty}(F):=\bigcap\limits_{n=1}^\infty Im(F^{n})  .$$ 
	Since $Im^{\infty} (F)$ is orthogonally complementable in  $\tilde{M},$ there exists orthogonal projection $P_{Im^{\infty}(F)^{\bot}}$ onto $Im^{\infty}(F)^{\bot}$ along  $Im^{\infty}(F)$ and $$(\ker F \cap Im^{\infty}(F))= \ker P_{Im^{\infty} (F)^{\bot}{{\mid}_{\ker F}}} .$$ 
	Since $\ker  F$ is self  dual being finitely generated, then by	\cite[Corollary 3.6.4]{MT},\\
	$\ker F\cap Im^{\infty} (F)$ is an orthogonal direct summand in $\ker F,$ so $$\ker F=(\ker F \cap Im^{\infty} (F)) \oplus N_{1}$$ 	for some closed submodule $N_{1}.$ Therefore $\ker F_{0}=\ker F \cap M$ is finitely generated being a direct summand in $\ker  F$ which  is finitely generated itself. Since $\ker F \cap M$ is finitely generated, by \cite[Lemma 2.3.7]{MT}, $\ker F \cap M$ is orthogonally complementable in $M,$ so $M=(\ker F \cap M ) \oplus M^{\prime}$ for some closed submodule $M^{\prime}.$
	On $M^{\prime},F_{0}$ is an isomorphism from $M^{\prime}$ onto $M,$ so $F_{0} \in  {\mathcal{M}\Phi} (M)$  (recall that $M=(\ker F \cap M) \oplus M^{\prime}),$ and $\ker F_{0}=\ker F \cap M,$ which is finitely generated).	
	By Lemma \ref{L03} , there exists  in $\epsilon_{2} >0$   such that  if $ \parallel \alpha \parallel < \epsilon_{2}\textrm{ , } \alpha \in G(\mathcal{A}) \cap Z(\mathcal{A}),$ then 
	$$\ker (F_{0}-\alpha I_{\mid M})\preceq \ker  F_{0}, Im(F_{0}-\alpha I_{\mid M})^{\bot}\preceq ImF_{0}^{\bot}$$ in $M$  and $$\textrm{index}(F_{0}-\alpha I)=\textrm{index} F_{0}=[\ker F_{0}]$$ since  $F_{0}$ is surjective. Since  $ImF_{0}^{\bot}= \lbrace 0 \rbrace $ (in $M$) as $F_{0}$ is surjective,
	$$Im(F_{0}-\alpha I)^{\bot}\textrm =0 \mbox{ for all } \parallel \alpha \parallel < \epsilon_{2},  \textrm{  } ,\alpha \in G(\mathcal{A}) \cap Z(\mathcal{A}),$$ 
	since $Im(F_{0}-\alpha I_{\mid_{M}})^{\bot} \preceq ImF_{0} ^{\bot}$ for all $\parallel \alpha \parallel <    \epsilon_{2}, \alpha \in G(\mathcal{A}) \cap Z(\mathcal{A}) .$\\
	Recall that $\ker (F-\alpha I) \subseteq Im^{\infty} (F)=M  .$ Therefore 
	$$[\ker (F-\alpha I)]=[\ker  (F_{0}-\alpha I_{\mid_{M}})]=\textrm{index}(F_{0}-\alpha I_{\mid_{M}})=\textrm{index} F_{0}=[\ker F_{0}]$$
	This holds whenever $\parallel \alpha \parallel < \epsilon_{2},  \textrm{  } ,\alpha \in G(\mathcal{A}) \cap Z(\mathcal{A}).$\\
	Now, $\ker F_{0}=\ker F \cap M$  and $\ker F=(\ker F \cap M ) \oplus N_{1}$. Therefore, if $\alpha \in G(\mathcal{A}) \cap Z(\mathcal{A})$ and $\parallel \alpha \parallel < \epsilon_{2},$ then
	$$[\ker F ]=[\ker F \cap M ]+[N_{1}]=[\ker F_{0}]+[N_{1}]= [ \ker (F-\alpha I)]+[N_{1}]$$
	whenever $ \parallel \alpha \parallel < \epsilon_{2}\textrm{ , } \alpha \in G(\mathcal{A}) \cap Z(\mathcal{A}).$ If, in addition $ \parallel \alpha \parallel < \epsilon_{1},$ then as we have seen in the beginning of this proof, by choice of $\epsilon_{1},$ we have $\textrm{index}(F-\alpha I)=\textrm{index}F.$\\
	So, if $ \parallel \alpha \parallel <min \lbrace \epsilon_{1} , \epsilon_{2} \rbrace$ for  $ \alpha \in G(\mathcal{A}) \cap Z(\mathcal{A}) ,$ then $\textrm{index}(F-\alpha I)=\textrm{index}F,$
	and $[\ker F]=[\ker (F-\alpha I)]+[N_{1}] .$ It follows that
	 $$[ImF^{\bot}]=[Im(F-\alpha I)^{\bot}]+[N_{1}]$$
\end{proof}
\begin{remark} \label{R06} %\textbf{R06} 
	If $\mathcal{A}$ is a factor, then Theorem \ref{T02}  is of interest in the case of finite factors, as $K(\mathcal{A})$ is trivial otherwise.	
\end{remark}

\textbf{Acknowledgement} I am especially grateful to my supervisor Professor Vladimir M. Manuilov  for careful reading of my paper and for  inspiring comments and suggestions that led to the improved presentation of the paper.
Also I am grateful to Professor Dragan S. Djordjevic for suggesting the research topic of the paper and for introducing to me the relevant reference books.

\end{document}